\documentclass[11pt,a4paper]{amsart}

\usepackage{amsmath,amsthm}
\usepackage{amssymb,esint}
\usepackage{amscd}
\usepackage{mathrsfs}
\usepackage{enumerate}
\usepackage{anysize}
\usepackage[english]{babel}
\usepackage[a4paper,top=2cm,bottom=2cm,left=3.5cm,right=3.5cm,marginparwidth=1.75cm]{geometry}

\usepackage{tikz}
\usepackage{subcaption}

\usetikzlibrary{fadings}
\usetikzlibrary{decorations}
\usepgflibrary{decorations.pathmorphing}

\newtheorem{theorem}{Theorem} 
\newtheorem{lemma}[theorem]{Lemma} 

\newtheorem{proposition}[theorem]{Proposition}

\newtheorem{remark}[theorem]{Remark}
\numberwithin{theorem}{section}
\numberwithin{equation}{section}

\def\eps{\varepsilon}

\def\la{\left\lvert}
\def\lA{\left\lVert}
\def\ra{\right\rvert}
\def\rA{\right\rVert}

\def\partialx{\nabla}

\def\a{\alpha}

\def\dalpha{\diff \! \alpha}
\def\dsigma{\diff \! \sigma}
\def\dt{\diff \! t}

\def\dh{\diff \! h}
\def\dx{\diff \! x}
\def\dxi{\diff \! \xi}
\def\dz{\diff \! z}
\def\dy{\diff \! y}
\def\dydx{\diff \! y \diff \! x}
\def\fract{\frac{\diff}{\dt}}

\def\BMO{\rm{BMO}}
\def\D{\la D_x\ra}

\def\slam{a}
\def\Slam{A}

\def\defn{\mathrel{:=}}
\def\eps{\varepsilon}
\def\la{\left\vert}
\def\lA{\left\Vert}
\def\bla{\big\vert}
\def\blA{\big\Vert}
\def\le{\leq}

\def\mez{\frac{1}{2}}
\def\ra{\right\vert}
\def\rA{\right\Vert}
\def\bra{\big\vert}
\def\brA{\big\Vert}
\def\tdm{\frac{3}{2}}
\def\tq{\frac{3}{4}}

\def\xN{\mathbb{N}}
\def\xR{\mathbb{R}}

\def\xT{\mathbb{T}}

\DeclareMathOperator{\cn}{div}
\DeclareMathOperator{\RE}{Re}

\DeclareMathOperator{\diff}{d}

\DeclareMathOperator{\Op}{Op}
\DeclareMathOperator{\cnx}{div}
\DeclareMathOperator{\curl}{curl}
\DeclareMathOperator{\pv}{pv}

\title{Paralinearization of free boundary problems in fluid dynamics}

\author{Thomas Alazard}
\thanks{Thomas Alazard, Universit\'e Paris-Saclay, ENS Paris-Saclay, CNRS, Centre Borelli UMR9010, 
avenue des Sciences, F-91190 Gif-sur-Yvette France, thomas.alazard@ens-paris-saclay.fr.}

\begin{document}

\maketitle

\begin{abstract} 
 A classical topic in the mathematical theory of hydrodynamics is to study the evolution of the free surface 
 separating air from an incompressible perfect fluid. The goal of this survey is to examine 
 this problem for two important sets of equations: the water wave equations and the Hele-Shaw equations, 
 including the Muskat problem. These equations are different in nature, 
dispersive or parabolic, but we will see that they can be studied using related tools. 
In particular, we will discuss a paradifferential approach to these problems.
\end{abstract}

\section{Introduction}

The aim of this survey article is to provide an overview of various results in the mathematical analysis of 
free boundary problems in fluid dynamics. 

The primary objective is to depict the evolution of a free surface, denoted as $\Sigma(t)$, which, 
by simplicity, is assumed to be represented as the graph of a 
function $\eta=\eta(t)$. Additionally, it is assumed that 
$\Sigma(t)$ undergoes transport driven by a fluid occupying the region $\Omega(t)\subset \xR^N$ 
located beneath $\Sigma(t)$. To be more specific, it is assumed that the normal velocity $V_n$ of $\Sigma$ 
obeys to the equation
\begin{equation}\label{I1}
V_n=u\cdot n,
\end{equation}
where $u\colon \Omega\to \xR^N$ is the fluid velocity, and $n$ represents 
the normal vector to the free surface. 

This is a broad formulation that encompasses various fluid dynamics problems. Indeed, numerous equations are associated with this problem due to various factors that influence the dynamics. These factors include the compressibility of the fluid, its irrotational or rotational nature, whether the fluid has a bottom or not, and whether the restoring forces are determined by gravity or surface tension. Furthermore, a range of equations can be derived under different asymptotic conditions. Notable examples include the Boussinesq and Korteweg-de Vries equations, which have been extensively studied 
(see~\cite{SW,ASL,IguchiCPDE,LannesLivre,Saut-IMPA}).

This paper 
primarily focuses on three fundamental models: the water wave problem, the Hele-Shaw equation, and the Muskat equation. These 
models differ in their mathematical nature (dispersive versus parabolic). However, we will show how related mathematical tools can be employed to study them. Specifically, this paper will discuss a paradifferential approach to these problems. In simple terms, 
the goal is to find a method to transform these equations in such a way that they can be reduced to simpler equations.

\subsection{The equations}
To study these equations, we adopt a formulation which involves a set of equations 
reduced to the free surface. This formulation revolves around the displacement function $\eta$ and the trace $\psi$ of the potential flow variable on the free surface. This particular approach has a rich historical background, notably tracing its roots back to the works of Zakharov~\cite{Zakharov}, Craig-Sulem~\cite{CrSu}, and Lannes~\cite{LannesJAMS}, who 
introduced an analysis based in the Dirichlet-to-Neumann 
operator. We adopt this perspective in the following.

Consider a time-dependent domain
\begin{align*}
&\Omega(t)=\{ (x,y) \in \xR^{d}\times \xR\,:\, y < \eta(t,x)\} \subset \xR ^N\qquad (N=d+1,~d\ge 1)
\intertext{and set}
&\Sigma(t)=\partial\Omega(t)=\{ y = \eta(t,x)\}.
\end{align*}
In this case, the normal $n$ and normal velocity $V_n$ are given by
\begin{figure}[h]
\begin{minipage}[b]{0.45\linewidth}
\begin{align*}
n&=\frac{1}{\sqrt{1+|\nabla \eta|^2}}\begin{pmatrix} -\nabla \eta\\ 1\end{pmatrix}\\[1ex]
V_n&=n\cdot \fract \begin{pmatrix} x \\ \eta(t,x)\end{pmatrix}=\frac{\partial_t\eta}{\sqrt{1+|\nabla \eta|^2}}.
\end{align*}
\end{minipage}
\hspace{5mm}
\begin{minipage}[b]{0.40\linewidth}
\begin{tikzpicture}[scale=0.9,samples=100]
\draw [thick,->] (-0.75,-0.3) -- (-0.5,0.5) ; 
\node at (-0.5,0.5) [right] {$n$}; 
\shadedraw [top color=blue!70!white, bottom color=blue!30!white, draw=white] 
(-3,0) -- (-3,-3) -- (1,-3) -- (1,0)  -- 
(1,-0.7) to [out=180,in=50]
(-2,-0.5) to [out=230,in=0] (-3,-1.2) -- (-3,-1) ;
\draw [black] (1,-0.7) to [out=180,in=50]
(-2,-0.5) to [out=230,in=0] (-3,-1.2) ;
\node at (0.5,0) [below] {$\Sigma(t)$};
\node at (-0.9,-1.7) {$\Omega(t)$};
\end{tikzpicture}
\end{minipage}
\end{figure}

Then, the kinematic equation
$V_n=u\cdot n$ (see~\eqref{I1}) simplifies to
\begin{equation}\label{N111}
\partial_t\eta=\sqrt{1+|\nabla \eta|^2}\, u\cdot n.
\end{equation}
Assume now that $u$ is irrotational and incompressible, that is assume that
$$
\curl_{x,y}u=0\quad,\quad \cnx_{x,y}u=0.
$$
Then $u=\nabla_{x,y}\phi$ for some potential $\phi\colon \Omega\to \xR$ such that 
$\Delta_{x,y} \phi = 0$. Now, following 
Zakharov~\cite{Zakharov1968,Zakharov}, set
$$
\psi(t,x)=\phi(t,x,\eta(t,x))
$$
and introduce, following Craig-Sulem \cite{CrSu} and Lannes~\cite{LannesJAMS,LannesLivre}, 
the Dirichlet-to-Neumann operator by 
$$
G(\eta)\psi =\partial_y \phi-\partialx \eta \cdot \partialx \phi \Big\arrowvert _{y=\eta}
=\sqrt{1+|\nabla\eta|^2}\, \partial_n\phi\Big\arrowvert _{y=\eta}.
$$
It follows from this definition and the kinematic equation~\eqref{N111} that 
\begin{equation}\label{N2}
\partial_t\eta=\sqrt{1+|\nabla \eta|^2}\, (u\cdot n)\arrowvert _{y=\eta}=\sqrt{1+|\nabla \eta|^2}\, \partial_n\phi \arrowvert _{y=\eta}=G(\eta)\psi.
\end{equation}

To derive a closed set of equations, we also require an evolution equation for $\psi$. 
The simplest way 
to obtain such an equation is to postulate that :
$$
\psi=-g\eta,
$$
where $g$ is the acceleration of gravity. 
Consequently, it follows from equation \eqref{N2} that
\begin{equation}\label{N3}
\partial_t\eta+gG(\eta)\eta=0.
\end{equation}
This equation has physical significance and is commonly referred to as the Hele-Shaw equation (as seen in references such as~\cite{ChangLaraGuillenSchwab,AMS,Nguyen-Pausader}). It corresponds to the physical scenario where the fluid velocity and pressure are related through Darcy's law, which states that:
$$
u=-\nabla_{x,y}(P+gy).
$$

Alternatively, one can consider a slightly more complex equation in which $\psi$ 
is determined by the solution of an equation of the form:
$$
\psi+K(\eta)\psi=-\eta.
$$
This case encompasses the significant Muskat equation, which corresponds to a Hele-Shaw problem with two phases (see~\cite{Nguyen-Pausader}).

Out of the numerous conceivable relationships between $\eta$ and $\psi$, 
the one that undoubtedly stands out as the most elegant is the discovery made by Vladimir Zakharov.

\begin{theorem}[{Zakharov} \cite{Zakharov1968}] Consider an  irrotational, incompressible velocity field 
$u=\nabla_{x,y}\phi$ satisfying 
$\partial_tu+u\cdot\nabla_{x,y}u=-\nabla_{x,y}(P+gy)$ and $P\arrowvert_{y=\eta}=0$.
Then $\eta$ and $\psi$ are conjugated: 
\begin{equation*}
\left\{
\begin{aligned}
\frac{\partial\eta}{\partial t}&=\phantom{-}\frac{\delta \mathcal{H}}{\delta \psi}\\
\frac{\partial\psi}{\partial t}&=-\frac{\delta \mathcal{H}}{\delta \eta}
\end{aligned}
\right.
\quad\text{where}\quad
\left\{
\begin{aligned}
&\psi(t,x)=\phi(t,x,\eta(t,x))\\[2ex]
&\mathcal{H}=\mez \int_{\xR^d} \psi G(\eta)\psi \dx + \frac{g}{2}\int_{\xR^d} \eta^2 \dx.
\end{aligned}
\right.
\end{equation*}
\end{theorem}

A popular form of the equations is the following :
\begin{equation}\label{N4}
\left\{
\begin{aligned}
&\partial_{t}\eta-G(\eta)\psi=0,\\
&\partial_{t}\psi+g \eta+ \frac{1}{2}\la\partialx \psi\ra^2  -\frac{1}{2}
\frac{\bigl(\partialx  \eta\cdot\partialx \psi +G(\eta) \psi \bigr)^2}{1+|\partialx  \eta|^2}
= 0.
\end{aligned}
\right.
\end{equation}
\begin{remark}
$i)$ It is proved in \cite{Bertinoro} that this system is equivalent to the incompressible 
Euler equations with free surface.

$ii)$ {Brenier} (\cite{brenier-hidden-lectures}) related the Hele-Shaw~\eqref{N3} and water-wave equations~\eqref{N4} by a quadratic change of time.

$iii)$ One can consider more general hamiltonians to take into account surface tension effect or 
rotational flows (see \cite{MR3869702,MR4088976} and references therein).
\end{remark}

A key feature of this problem is that~\eqref{N4} is not a system of PDEs (indeed, $G(\eta)$ is a nonlocal operator). 
Other difficulties present themselves: the equations are fully nonlinear (instead of semi-linear, see the thesis by Ayman Rimah Said \cite{Ayman2}). Eventually, 
the Hamiltonian does not control the dynamics.

\section{The Dirichlet-to-Neumann operator}
One can define 
the Dirichlet-to-Neumann operator 
for rough domains. For instance,  
Arendt and ter Elst (\cite{MR2823661}) define it for bounded connected open set $\Omega\subset\xR^N$ 
whose boundary has a finite $(N-1)$-dimensional Hausdorff measure (see also~\cite{taylor2020DtN}). 
Here, we will consider smoother functions. 
In broad terms, we aim to identify a level of regularity beyond which we can effectively transform the equations, enabling us to conjugate them to simpler counterparts. 
We begin by discussing classical results that hold for Lipschitz free surface. 
Subsequently, we will consider free surfaces with slightly higher regularity, specifically those belonging to the class $C^{1,\alpha}$ in 
terms of Sobolev embedding.

\subsection{The Lipschitz threshold of regularity}

Assume $\eta\in W^{1,\infty}(\xR^d)$ 
and $\psi \in H^{\mez}(\xR^d)$. 
Then it follows from classical arguments from functional analysis (Poincar\'e type inequalities and 
the Lax-Milgram theorem, see \cite{ABZ1,ABZ3,ABZ4}) that there is a unique variational solution 
$$
\phi \in L^2\big(\dx\dy /(1+|y|)^2\big)\quad,\quad \nabla_{x,y}\phi \in L^2(\Omega) 
$$
to
\begin{equation}\label{def:phi}
\Delta_{x,y} \phi = 0 \quad\text{in } \{(x,y)\in\xR^d\times \xR\,:\, y<\eta(x)\}, 
\qquad \phi\big\arrowvert_{y=\eta}= \psi.
\end{equation}
Given that $\phi$ is harmonic, it is possible to express the normal 
derivative in terms of the tangential 
derivatives. By doing so, we verify that its normal derivative on the free surface 
is well-defined 
as an element of $H^{-\mez}(\xR^d)$. Therefore, one can define 
$$
G(\eta)\psi =\sqrt{1+|\nabla\eta|^2}\,
\partial_{n}\phi\big\arrowvert _{y=\eta}\in H^{-\mez}(\xR^d),
$$
and it follows that
$$
\lA G(\eta)\rA_{H^\mez\to H^{-\mez}}\le C\big( \lA \nabla \eta\rA_{L^\infty}\big).
$$
Moreover, a famous inequality of Rellich 
shows that $G(\eta)$ is also well-defined on $H^1(\xR^d)$ 
and that one has
\begin{equation}\label{Rellich-original}
\lA G(\eta)\rA_{H^1\to L^2}\le C\big( \lA \nabla \eta\rA_{L^\infty}\big).
\end{equation}

The first inequality of the form \eqref{Rellich-original} was established through an integration by parts argument by Rellich \cite{MR0002456}. Rellich's goal was to investigate the eigenvalues of the Laplacian within star-shaped domains. This inequality plays a pivotal role in numerous works  pertaining to elliptic partial differential equations. Notably, such inequalities were used by Jerison and Kenig to study the Dirichlet problem in Lipschitz domains \cite{JeKe81Dirichlet} 
and by Verchota \cite{Verchota84} in his famous work on layer potentials. 
See the books~\cite{McLean,Necas} for many other applications. 

Below, we present a concise (formal) proof of a refined version of this inequality, which follows from Zakharov's Hamiltonian formulation and Noether's theorem.

\begin{proposition}[Noether's theorem implies Rellich inequality]
For any $d\ge 1$, 
$$
\int_{\xR^d} \bla \nabla_{x,y}\phi  \arrowvert_{y=\eta}\bra^2\dx
\le 5 \int_{\xR^d}|\nabla \psi|^2\dx.
$$
In particular
$$
\lA G(\eta)\rA_{H^1\to L^2}\le 4+4\lA \nabla \eta\rA_{L^\infty}.
$$
\end{proposition}
\begin{proof}This generalizes an inequality proved in \cite{MR4556204} 
for the case $d=1$ using the multiplier method. Here we give a short (formal) proof which relies only on 
Noether's theorem.

It follows from the Hamiltonian formulation of the problem and its symmetries that
$$
\fract \int \eta(t,x)\dx=0 \quad \text{and}\quad \fract \int \psi(t,x)\dx=0.
$$
Now, remembering that
$$
\partial_{t}\psi+g \eta+ \frac{1}{2}\la\partialx \psi\ra^2  -\frac{1}{2}
\frac{\bigl(G(\eta) \psi +\partialx  \eta\cdot\partialx \psi \bigr)^2}{1+|\partialx  \eta|^2}=0
$$
we infer that 
$$
\int \frac{\bigl(G(\eta) \psi +\partialx  \eta\cdot\partialx \psi \bigr)^2}{1+|\partialx  \eta|^2}\dx
=\int \la\partialx \psi\ra^2\dx.
$$
(We parenthetically mention that the previous identity can be rigorously 
derived from the divergence theorem, see~\cite{MR4556204}.) 
Now, by using the chain rule, we verify that
$$
(\partial_y\phi)\arrowvert_{y=\eta}=\frac{G(\eta) \psi +\partialx  \eta\cdot\partialx \psi }{1+|\partialx  \eta|^2}.
$$
Consequently, the previous estimate implies that
\begin{equation}\label{N200}
\int (1+|\partialx  \eta|^2)\big(( \partial_y\phi)\arrowvert_{y=\eta}\big)^2\dx
=\int \la\partialx \psi\ra^2\dx.
\end{equation}
On the other hand, using again the chain rule, we have
$$
(\nabla_x\phi)\arrowvert_{y=\eta}=\nabla\psi-(\partial_y\phi)\arrowvert_{y=\eta}\nabla \eta,
$$
and hence the estimate of $(\nabla_x\phi)\arrowvert_{y=\eta}$ follows immediately from~\eqref{N200}. 
\end{proof}

We next give an elementary proof (obtained with Quoc-Hung Nguyen in 2021, see~\cite{Berkeley}) of an optimal trace inequality. 
\begin{proposition}[Alazard-Nguyen]
Let $d\ge 1$, 
$\eta\in W^{1,\infty}(\xR^d)$ and $\psi\in H^\mez(\xR^d)$. There exists $c>0$ such that
\begin{equation}\label{I2}
\int_{\xR^d} \psi G(\eta)\psi\dx 
\ge \frac{c}{1+\lA \nabla \eta\rA_{\BMO}} \lA \psi\rA_{\dot{H}^\mez}^2.
\end{equation}
\end{proposition}
\begin{remark}[A trace inequality]
Let $u\in H^1(\Omega)$ and set $\psi=u\arrowvert_{y=\eta}$. Denote by $\phi$ the harmonic extension of $\psi$. Then
\begin{align*}
\iint_{\Omega} \la\nabla_{x,y}u\ra^2\dydx\ge \iint_{\Omega} \la\nabla_{x,y}\phi\ra^2\dydx=\int_{\partial\Omega}\phi\partial_n\phi\dsigma,
\end{align*}
where we used the divergence theorem to obtain the last equality. Now, by definition of $G(\eta)$, we have
\begin{align*}
\int_{\partial\Omega}\phi\partial_n\phi\dsigma
=\int_{\xR^d} \psi G(\eta)\psi\dx.
\end{align*}
Consequently, it follows from~\eqref{I2} that
$$
\iint_{\Omega} \la\nabla_{x,y}u\ra^2\dydx\ge \frac{c}{1+\lA \nabla \eta\rA_{\BMO}} \lA u\arrowvert_{y=\eta}\rA_{\dot{H}^\mez}^2.
$$
\end{remark}
\begin{remark}[An optimal inequality] As explained in \cite{AZ-2023virial},  
the dependence in $\lA\nabla \eta\rA_{\BMO}$ is \emph{optimal}. 
More precisely, it is proved in the latter reference that
$$
\int \psi G(\eta)\psi\dx \ge \frac{c}{(1+\lA \nabla \eta\rA_{L^\infty})^{m}} \lA \psi\rA_{\dot{H}^\mez}^2\quad \Rightarrow\quad m\ge1.
$$
\end{remark}

\begin{proof}
For the sake of notational simplicity, we consider only the case $d=1$ (we refer to~\cite{Berkeley} for the general case). 
Set $v(x,z)=\phi(x,z+\eta(x))$. Then, it follows from the divergence theorem that
$$
\int_\xR \psi G(\eta)\psi\dx =\iint_{\Omega} \la\nabla_{x,y}\phi\ra^2\dydx = \iint_{\xR^{2}_-} \Big[  (\partial_x v-\partial_z v\partial_x \eta)^2+
(\partial_z v)^2\Big] \dz\dx=:Q.
$$
On the other hand, the Plancherel theorem implies that
\begin{align}
&\int_\xR \psi \D \psi\dx=\iint_{\xR^{2}_-} \partial_z (v \D v)\dz\dx
=2\iint_{\xR^{2}_-}(\partial_z v)\D v\dz\dx\notag\\
&\qquad=2\iint_{\xR^{2}_-}\Big[(\partial_z v)\mathcal{H}\big(\partial_x v-
(\partial_z v)\partial_x \eta\big)+
(\partial_z v)\mathcal{H}\big(
(\partial_z v)\partial_x \eta\big)\Big]\dz\dx,\label{N1}
\end{align}
where we denoted by $\mathcal{H}$ the Hilbert transform. 
Since $\mathcal{H}$ is bounded on $L^2$, we have
$$
2\iint_{\xR^{2}_-}(\partial_z v)\mathcal{H}\big(\partial_x v-
(\partial_z v)\partial_x \eta\big)\dz\dx
\le Q=\int_\xR \psi G(\eta)\psi\dx.
$$
So it remains only to estimate the second term in the left-hand side of \eqref{N1}. To do this, observe that, 
since $\mathcal{H}^*=-\mathcal{H}$, we have
$$
2\iint_{\xR^{2}_-}(\partial_z v)\mathcal{H}\big(
(\partial_z v)\partial_x \eta\big)\dz\dx
=\iint_{\xR^{2}_-}  (\partial_z v) \big[\mathcal{H},\partial_x \eta\big]\partial_z v \dz\dx.
$$
Then it suffices to recall that a classical result by Coifman-Rochberg-Weiss \cite{MR0412721} which states that the commutator between the Hilbert transform and a $\BMO$ function is bounded on $L^2$ (see also \cite{MR4062964}).
\end{proof}

\subsection{Elliptic estimates}

We have seen that
$$
\lA G(\eta)\rA_{H^\mez\to H^{-\mez}}+\lA G(\eta)\rA_{H^1\to L^2}\le C( \lA \nabla \eta\rA_{L^\infty}).
$$
Many other results have been proved when $\eta$ is more regular (\cite{Nalimov,CSS,WuInvent,BG,LannesJAMS}). 
In particular, we have the following

\begin{proposition}[from~\cite{ABZ3}]

$(i)$ For all $s>1+d/2$ and $1/2\le \sigma\le s$
$$
\lA G(\eta)\psi\rA_{H^{\sigma-1}} \le C\left( \lA \eta\rA_{H^{s}}\right) \lA \psi\rA_{H^\sigma}.
$$ 

$(ii)$ For all $s>1+d/2$,
\begin{equation*}
\lA \left[ G(\eta_1)-G(\eta_2)\right] \psi\rA_{H^{s-\tdm}}\le 
C\left(\lA (\eta_1,\eta_2)\rA_{H^{s+\mez}}\right) \lA \psi\rA_{H^{s}}
\lA \eta_1-\eta_2\rA_{H^{s-\mez}}.
\end{equation*}
\end{proposition}

Notice that, in terms of Sobolev embedding, case $(i)$ corresponds to a function $\eta$ that is slightly more regular than Lipschitz. However, for case $(ii)$, the difference $\eta_1 - \eta_2$ is estimated in a norm weaker than the Lipschitz norm.

One possible approach to prove such Sobolev estimates is to employ an elementary change of variables, reducing the problem to a case where the domain is a half-space. This transformation replaces the Laplace equation with an elliptic equation with variable coefficients (see~\eqref{i29} below), depending on the derivatives of $\eta$ (which are H\"older continuous in terms of Sobolev embedding). A classical idea, rooted in Schauder's method, involves freezing the coefficient at a point, expressed as
$$
\cnx_{x,z}(\underbrace{A(x_0,z_0)}_{{\text{constant}}}\nabla_{x,z}v)=\cnx_{x,z}(\underbrace{(A(x_0,z_0)-A(x,z))}_{{\text{small}}}\nabla_{x,z}v).
$$ 
There are several ways to rigorously justify this. One method, explained in the next section, involves the use of paradifferential calculus.

\subsection{The shape derivative}

Eventually, we recall the following fundamental identity  which
helps to understand the dependence on~$\eta$.

\begin{proposition}[Lannes' shape derivative formula]\label{P:shape}
For all $\eta,\psi,\zeta\in H^\infty(\xR^d)$, the shape derivative of the Dirichlet-to-Neumann operator is given by
 \begin{equation}\label{n:shape}
  \lim_{\eps\rightarrow 0} \frac{1}{\eps}\bigl( G(\eta+\eps \zeta)\psi -G(\eta)\psi\bigr)
  = -G(\eta)(B\zeta) -\cnx (V\zeta),
 \end{equation}
where $V=(\nabla\phi)\arrowvert_{y=\eta}$ and 
$B=(\partial_y \phi)\arrowvert_{y=\eta}$, with $\phi$ the harmonic extension of $\psi$ solution to~\eqref{def:phi}.
\end{proposition}
\begin{remark}
This identity extends to functions with limited regularity; 
we refer to the Lannes' original article or his book (see~\cite{LannesJAMS,LannesLivre}).
\end{remark}

\section{Paralinearization of the Dirichlet-to-Neumann operator}

\subsection{Pseudo- and paradifferential operators}

Since the work of Kohn--Nirenberg and H\"ormander it is said that 
that a linear operator $T\colon \mathcal{S}'(\xR^d)\to\mathcal{S}'(\xR^d)$ acting on temperate distributions is a pseudo-differential operator if we can define it 
from a 
function $a=a(x,\xi)$ by the relation
\begin{equation}\label{Quantization1}
T \bigl(e^{ix\cdot \xi}\bigr) = a(x,\xi)e^{ix\cdot \xi}.
\end{equation}
We then say that $a$ is the symbol for $T$ and we denote $T=\Op(a)$. 
For instance, the operator 
associated with the symbol $a=\sum_{\alpha}a_\alpha(x)(i\xi)^\alpha$ is simply 
the differential operator 
$T=\sum_{\alpha}a_\alpha(x)\partial_x^\alpha$ (with classical notations). 
Another fundamental example is the case of the operator $\la D_x\ra$ defined by
$$
\la D_x\ra e^{ix\cdot \xi}=\la \xi\ra e^{ix\cdot \xi},
$$
so that $\la D_x\ra=\Op(\la \xi\ra)$.

The pseudo-differential calculus is a process that associates a symbol $a=a(x,\xi)$ defined on $\mathbb{R}^d\times \mathbb{R}^d$ with an operator $\Op (a)$, in a way which allows for an understanding of the properties of $\Op (a)$ (such as product, adjoint, and boundedness on the usual spaces of functions) simply by examining the properties of the symbols.

The application that associates an operator $\Op(a)$ to the symbol $a$ is called a quantization. There are numerous known quantizations, which are variants of \eqref{Quantization1}. Bony's quantization is particularly well-suited for nonlinear problems. Its specificity lies in the quantization of symbols with limited regularity in $x$. Bony's theory of paradifferential operators enables the study of the regularity of solutions to nonlinear partial differential equations. A brief introduction is provided here, with references to the original article by Bony~\cite{Bony}, as well as the books by H\"ormander~\cite{Hormander}, M\'etivier~\cite{MePise}, Meyer~\cite{Meyer}, and Taylor~\cite{MR1121019} for a more comprehensive understanding of the general theory.

The idea of studying the incompressible Euler equation at the free surface using tools derived from the analysis of singular integrals dates back to an article by Craig-Schanz-Sulem~\cite{CSS}. This concept has been further developed by Lannes~\cite{LannesJAMS} and Iooss-Plotnikov-Toland~\cite{IPT,IP}. The paradifferential analysis of the Dirichlet-to-Neumann operator is introduced by Alazard-M\'etivier in~\cite{AM} and further developped by Alazard-Burq-Zuily~\cite{ABZ1,ABZ3,MR4484254} and Alazard-Delort~\cite{AlDe-Sob} (see also Remark~\ref{R:refs}).

\subsection{Calder\'on's analysis of the Dirichlet-to-Neumann operator}

By using the Fourier transform, it is easily seen that 
$G(0)= \la D_x\ra$. More generally, if 
$\eta\in C^\infty$, then it is known since 
Calder\'on~(\cite{Calderon63}) that $G(\eta)$ is an elliptic, self-adjoint 
pseudo-differential operator of order $1$, 
whose principal  symbol is
\begin{equation}\label{N30}
\lambda(x,\xi)\defn\sqrt{(1+\la\partialx\eta(x)\ra^2)\la\xi\ra^2-(\partialx\eta(x)\cdot\xi)^2}.
\end{equation}
This means that one can write $G(\eta)$ under the form
$$
G(\eta)\psi =
(2\pi)^{-d}\int e^{ix\cdot\xi} \lambda(x,\xi)\widehat{\psi}(\xi)\dxi 
+ R(\eta)f,
$$
where the remainder $R(\eta)$ is of order $0$, satisfying (see Lannes~\cite{LannesLivre})
$$
\exists K\ge 1,\forall s\ge \mez, \quad 
\lA R(\eta)\psi\rA_{H^{s}}\le C\left( \lA \eta\rA_{H^{s+K}}\right)\lA \psi\rA_{H^s}.
$$

\begin{remark} Three remarks are in order:

$i)$ Notice that $\lambda$ is well-defined for any $\eta\in W^{1,\infty}(\xR^d)$.

$ii)$ If $d=1$ or $\eta=0$ then $\lambda(x,\xi)=\la \xi\ra$ and $\Op(\lambda)=\la D_x\ra$. 

$iii)$ The constant $K$ corresponds to a \emph{loss of derivatives} (see \cite{Iguchi01,LannesJAMS}). 
\end{remark}

In dimension one, the fact that $\lambda=\la \xi\ra$ implies that
\begin{align*}
G(\eta)\psi
=\la D_x\ra \psi ~+~R(\eta)\psi.
\end{align*}

By using paradifferential calculs, we will see that one can describe the remainder term $R(\eta)=G(\eta)\psi-\la D_x\ra\psi$.

\subsection{Paraproduct}\label{S:3.3}

Paraproducts are well-known for their crucial role in the study of multilinear Fourier multipliers (see~\cite{MeCo,Tao-Book}). Loosely speaking, paraproducts are restricted versions of products in which only certain types of frequency interactions are permitted. A paraproduct can be defined very simply using the Fourier inversion formula:
$$
a(x)b(x)=\frac{1}{(2\pi)^{2d}}\iint_{\xR^d\times\xR^d} e^{ix\cdot (\xi_1+\xi_2)} \, \widehat{a}(\xi_1)\, \widehat{b}(\xi_2)\dxi_1 \dxi_2.
$$
Let us decompose the integral in three terms
$$
\iint_{\xR^d\times \xR^d}=\iint_{\la\xi_1+\xi_2\ra\sim \la\xi_2\ra}+\iint_{\la\xi_1+\xi_2\ra\sim \la\xi_1\ra}
+\iint_{\la\xi_1\ra\sim \la\xi_2\ra}
$$
to obtain the following decomposition of the product
$$
ab =T_a b+T_b a+R(a,b).
$$

To define these operators more precisely, let us consider a cut-off function $\chi
$ such that
$$
\chi
(\xi_1
,\xi_2
)=1 \quad \text{if}\quad \la\xi_1
\ra\le \eps_1\la \xi_2
\ra,\qquad 
\chi
(\xi_1
,\xi_2
)=0 \quad \text{if}\quad \la\xi_1
\ra\geq \eps_2\la\xi_2
\ra,
$$
with $0<\eps_1<\eps_2<1$. Given two functions $a=a(x)$ and $b=b(x)$ one defines
\begin{align*}
T_a b&=
\frac{1}{(2\pi)^{2d}}\iint 
e^{ix\cdot(\xi_1+\xi_2)}
\chi(\xi_1,\xi_2)\widehat{a}(\xi_1)\widehat{b}(\xi_2)\dxi_1 \dxi_2,\\
T_b a&=\frac{1}{(2\pi)^{2d}}
\iint e^{ix\cdot(\xi_1+\xi_2)}\chi
(\xi_2,\xi_1)\widehat{a}(\xi_1)\widehat{b}(\xi_2)\dxi_1 \dxi_2,\\
R(a,b)&=\frac{1}{(2\pi)^{2d}}\iint e^{ix\cdot(\xi_1
+\xi_2
)}\bigl(1-\chi
(\xi_1
,\xi_2
)-\chi
(\xi_2
,\xi_1
)\bigr)\widehat{a}(\xi_1
)\widehat{b}(\xi_2
)\dxi_1 \dxi_2.
\end{align*}

This is Bony's decomposition of the product of two functions. 
One says that $T_a b$ and $T_b a$ are paraproducts, while 
$R(a,b)$ is a remainder. The key property is that a paraproduct 
by an $L^\infty$ function acts on any Sobolev spaces $H^s$ with $s$ in~$\xR$. 
The remainder term $R(a,b)$ is smoother than the paraproducts $T_a b$ and $T_b a$. More precisely, we have the following results:

\begin{theorem}[Bony]\label{TBony} 
For all $d\ge 1$ and all $s\in \xR$, we have
\begin{align*}
&&a\in L^\infty(\xR^d),&&u\in H^{s} (\xR^d)
&&\Rightarrow &&T_a u\in H^{s}(\xR^d).
\intertext{In addition, for all $s>0$, there holds}
&&a\in H^s(\xR^d),&&u\in H^s(\xR^d) 
&&\Rightarrow &&R(a,u)\in H^{{2s}-\frac{d}{2}}(\xR^d).
\end{align*}
\end{theorem}

\begin{theorem}[Coifman--Meyer]
If $a\in W^{1,\infty}(\xR^d)$ and $u\in L^2(\xR^d)$ then 
$$
au-T_a u \in H^1(\xR^d).
$$
\end{theorem}

\subsection{Paralinearization formula}

We are now ready to state a paralinearization formula for the Dirichlet-to-Neumann operator; more precisely a paradifferential refinement of the linearization formula given by Proposition~\ref{P:shape}. For the sake of simplicity, 
we consider only the case $d=1$ which allows to state a result involving 
only paraproducts. 
The case $d\ge 2$ requires in addition to introduce paradifferential operators.

\begin{theorem}[from~\cite{AM}]\label{T13} Let $d=1$ and $3<\gamma<s$. Then
\begin{equation}\label{Npara}
G(\eta)\psi=\D (\psi-T_B \eta)
-\partial_x \bigl( T_{V}\eta\bigr)+F
\end{equation}
where $V=(\partial_x\phi)\arrowvert_{y=\eta}$, 
$B=(\partial_y \phi)\arrowvert_{y=\eta}$ ($\phi$ is given by \eqref{def:phi}) and $F$ is a smooth quadratic remainder:
\begin{equation}\label{I7}
\lA F\rA_{H^{s+\gamma-4}}
\le 
C\left(\lA \eta\rA_{C^{\gamma}}\right)\left\{
\lA \psi\rA_{C^\gamma}\lA \eta\rA_{H^s}+\lA \eta\rA_{C^{\gamma}}\blA \psi \brA_{H^{s}}\right\}.
\end{equation}
\end{theorem}

\begin{remark}\label{R:refs}
Many extensions of this results have been proved by many authors. To name a few, we 
refer to Alazard-Burq-Zuily~\cite{ABZ1,ABZ3}, Alazard-Delort~\cite{AlDe-Sob}, 
de Poyferr\'e~\cite{dePoyferre1}, de Poyferr\'e-Nguyen~\cite{dPN1,dPN2}, Berti-Delort~\cite{BD-2018}, Ai~\cite{Ai-2019,Ai-2020}, Wang~\cite{Wang-2020}, Zhu~\cite{zhu2022propagation}, 
Fan Zheng~\cite{MR4400907}, Chengyang Shao~\cite{shao2023toolbox}; additional references are given in Section~\ref{S:4.1}.
\end{remark}

The statement involves only paraproducts, but we will explain that the proof involves, in an essential way, two other kinds of operators: paracomposition operators and paradifferential operators. 
Before that, we begin by analyzing a simple case where we consider only the quadratic terms in the 
Taylor expansion of the Dirichlet-to-Neumann operator.

\subsection{A toy quadratic model.}

We here study the quadratic terms in the Taylor expansion 
of the Dirichlet-Neumann operator~$G(\eta)$ 
with respect to the free surface elevation~$\eta$. 
Craig, Schanz and Sulem (see \cite{CSS} and~\cite[Chapter~$11$]{SuSu}) have shown that one can expand the 
Dirichlet-to-Neumann operator as a sum of pseudo-differential operators and gave 
precise estimates for the remainders. This Taylor expansion can be obtained by using Lannes' shape derivative formula. In particular, we have the following result (proved 
in Chapter $2$ in~\cite{AlDe-Sob}).

\begin{proposition}\label{T21}
Set
$$
G_{\le 2}(\eta)f= \la D_x\ra f
{-\la D_x\ra(\eta \la D_x\ra f) 
-\partial_x(\eta\partial_xf)}.
$$
Let~$(s,\gamma,\mu)\in \xR^3$ be such that 
$$
s-1/2>\gamma\ge 14,\quad s\ge \mu\ge 5,\quad \gamma\not\in \mez \xN,
$$ 
and 
consider~$(\eta,\psi)\in H^{s+\mez}(\xR)\times H^{s}(\xR)$. 
There exists a non decreasing function~$C\colon \xR\rightarrow \xR$ such that, \begin{multline}\label{n136}
\lA G(\eta)\psi-G_{\le 2}(\eta)\psi\rA_{H^{\mu-1}}\\ \le 
C(\lA \eta\rA_{C^{\gamma}} )
\lA \eta\rA_{C^{\gamma}} \Bigl\{\lA \psi\rA_{C^\gamma}\lA \eta \rA_{H^s}+\lA \eta \rA_{C^{\gamma}}
\blA \psi\brA_{H^{\mu+\mez}}\Bigr\}.
\end{multline}
\end{proposition}

We will now paralinearize the quadratic term and prove Theorem~\ref{T13} for this model case. Let us write $\la D_x\ra(\eta \la D_x\ra f) +
\partial_x(\eta\partial_xf)$ as a bilinear Fourier multiplier:
\begin{align*}
\la D_x\ra(\eta \la D_x\ra f) +
\partial_x(\eta\partial_xf)
&=\frac{1}{(2\pi)^{2}}\iint e^{ix(\xi_1+\xi_2)}
m(\xi_1,\xi_2)\widehat{\eta}(\xi_1)\widehat{f}(\xi_2)\dxi_1\dxi_2\\[2ex]
m(\xi_1,\xi_2)&=\la \xi_1+\xi_2\ra\la \xi_2\ra-(\xi_1+\xi_2)\xi_2.
\end{align*}
Now the key observation is that
$$
\la \xi_1\ra< \la \xi_2\ra 
~\Rightarrow~ (\xi_1+\xi_2)\xi_2>0 ~\Rightarrow~ m(\xi_1,\xi_2)=0.
$$
Consequently
$$
\la D_x\ra(T_\eta \la D_x\ra f) +
\partial_x(T_\eta \partial_xf)=0
$$
and hence
\begin{align*}
G_{\le 2}(\eta)f 
&= \la D_x\ra f 
-\la D_x\ra T_{| D_x|f}\eta 
-\partial_x (T_{\partial_x f} \eta)
-\la D_x\ra R(| D_x|f,\eta) 
-\partial_x R(\partial_x f, \eta)\\
&
=\la D_x\ra (f-T_{| D_x|f}\eta)
-\partial_x (T_{\partial_x f} \eta)
+F_2\quad\text{where}\\
F_2&=-\la D_x\ra R(| D_x|f,\eta) 
-\partial_x R(\partial_x f, \eta).
\end{align*}
Now, directly from Theorem~\ref{TBony}, 
we see that $F_2$ satisfies an estimate analogous to \eqref{I7}.

\subsection{The good unknown of Alinhac}\label{S:3.6}

In this section, we aim to explain the significance of the quantity
\[
\omega = \psi - T_{B}\eta
\]
which plays a crucial role in the paralinearization formula~\eqref{Npara} for the Dirichlet-to-Neumann operator. This unknown 
was initially introduced by Alazard-M\'etivier in \cite{AM} where it was 
associated with the concept of the "good unknown" of Alinhac. The latter plays a pivotal role in the analysis of surface waves, particularly in the work of Alinhac on shock waves and rarefaction waves for 
systems of conservation laws \cite{Ali,Alipara} (we refer to  Majda \cite{MajMem1,MajMem2} and M\'etivier \cite{Metivier1989}). It is noteworthy that, in the context of water-wave problems, a variation (at the linearized equation level, rather than the paradifferential level) was already 
employed by Lannes \cite{LannesJAMS} and Trakhinin \cite{Trakhinin}. 

Here, we aim to clarify that the appearance of $\omega$ is inherent when introducing the paracomposition operator of Alinhac \cite{Alipara}, which is associated with the change of variables that flattens the boundary $\{y=\eta(x)\}$ of the domain. This approach optimally captures the limited smoothness of the coordinate transformation.

To study the elliptic equation $\Delta_{x,y}\phi=0$ in $\Omega=\{(x,y)\in \xR^{d+1}\, 
;\, y<\eta(x)\}$, a classical strategy is to reduce the problem to the 
half-space $\xR^d\times (-\infty,0)$. To do so, the simplest  
change of coordinates is
$$
\kappa\colon (x,z)\mapsto (x,z+\eta(x)).
$$ 
Then $\phi(x,y)$ solves the Laplace equation $\Delta_{x,y}\phi=0$ if and only if $v=\phi\circ \kappa=\phi(x,z+\eta(x))$ is a solution of 
$Pv=0$ in $z<0$, where the operator $P$ is defined by
\begin{equation}\label{i29}
P=(1+|\nabla \eta|^2)\partial_z ^2+\Delta^2-2\nabla\eta\cdot\nabla\partial_z-(\Delta\eta)\partial_z.
\end{equation}
The boundary condition on $\phi\arrowvert_{y=\eta(x)}$ 
becomes $v(x,0)=\psi(x)$ and $G(\eta)$ is given by
$$
G(\eta)\psi=\bigl[ (1+|\nabla\eta|^2)\partial_zv-\nabla\eta\cdot\nabla v\bigr]\big\arrowvert_{z=0}.
$$

We can now explain the main difficulty to work with a diffeomorphism 
with limited regularity. 
Denote $D=-i\partial$ and 
introduce
$$
p(x,\xi,\zeta)=(1+|\nabla\eta(x)|^2)\zeta^2+|\xi|^2-2\nabla\eta(x)\cdot \xi \zeta
+i(\Delta\eta(x))\zeta.
$$
Observe that 
$P=-p(x,D_x,D_z)$ and write  
$T_{p}$ as a short notation for 
$$
T_{1+|\nabla(x)|^2}D_z^2+|D_x|^2-2T_{\nabla\eta}\cdot D_xD_z
+T_{\Delta\eta}\partial_z.
$$ 
Since $p(x,D_x,D_z)v=0$, by using Theorem~\ref{TBony}, we get hat 
$T_{p}v=f_1$ where 
$f_1$ is 
continuous in $z$ with values in $H^{s-2}$ if $\eta$ is in $H^s$ 
and if in addition 
the first and second order derivatives in $x,z$ of $v$ are 
bounded. 

Now, the main observation 
is that 
one can associate to the diffeomorphism $\kappa$ a paracomposition 
operator, denoted by $\kappa^*$, such that 
$T_{p}( \kappa^* \phi)=f_2$ 
where $f_2$ is a smoother source term. More precisely, where 
$f_2$ is continuous in $z$ with values in $H^{s+\gamma-K(d)}$, 
if $\eta$ is in $H^s$ and if the derivatives in $x,z$ 
of order 
less than $\gamma$ of $v$ are bounded 
(the main difference between $f_1$ and $f_2$ is that one cannot obtain a better 
regularity for $f_1$ by assuming that $v$ is smoother). 

We do 
not define $\kappa^*$ here. Instead let us describe 
the two main properties of paracomposition operators 
(see the article~\cite{Alipara}).
First, 
up to a smooth remainder, 
there holds
$$
\kappa^*\phi =\phi\circ \kappa-T_{\phi'\circ \kappa}\kappa
$$
where we denote by $\phi'$ the differential of $\phi$. 
On the other hand, one can use a symbolic calculus 
formula to compute the commutator of $\kappa^*$ with a paradifferential operator. 
The latter formula implies that
$$
\kappa^* \Delta -T_{p} \kappa^* 
$$
is a smoothing operator (which means here an operator bounded from $H^\mu$ to 
$H^{\mu+m}$ for any real number $\mu$, where 
$m$ is a fixed positive number depending on the regularity of $\kappa$). 
Since $\Delta_{x,y}\phi=0$, this means that 
$T_{p} \bigl(\phi\circ \kappa-T_{\phi'\circ \kappa}\kappa\bigr)$ is a smooth remainder term as claimed above. 

Now observe that
$$
\omega=\bigl( \phi\circ \kappa-T_{\phi'\circ \kappa}\kappa\bigr)\big\arrowvert_{z=0}.
$$
This explains why the good unknown of Alinhac $\omega$ enters into the analysis of free boundary problems.

\subsection{Elliptic factorization}
Inspired by the theory of Alinhac's paracomposition operators, let us introduce
$$
u=\phi\circ \kappa -T_{\phi'\circ \kappa}\kappa =v - T_{\partial_{z}v} \eta, 
\quad\text{where}\quad v(x,z)=\phi(x,z+\eta(x)).
$$
As explained in \S\ref{S:3.6}, 
$u$ satisfies a paradifferential elliptic equation with a smooth remainder term:
\begin{equation*}
T_{1+\la \partialx\eta\ra^2}\partial_z^2 u  +\Delta u -2 T_{\partialx\eta}\cdot \partialx \partial_z u  
- T_{\Delta \eta}\partial_z u\sim 0
\end{equation*}
where we say that $f\sim 0$ provided that 
$$
f\in C_z^0\big([-1,0];H_x^{{s+\gamma}-K(d)}(\xR^d)\big)
$$
for some constant $K$ depending only on the dimension $d$ (we stated Theorem~\ref{T13} only for $d=1$, but we explain here the strategy of the proof in arbitrary dimension~$d$). 

To prove Theorem~\ref{T13}, the next step is to 
compute the normal derivatives of $u$ at the boundary in terms of tangential derivatives. To do so, we adapt the strategy 
introduced by Calder\'on~\cite{Calderon63} to this paradifferential setting. 
This requires to introduce paradifferential operators. 

Given a symbol $a=a(x,\xi)$, 
the paradifferential operator $T_a$ of {Bony} is defined by symbol smoothing:
\begin{equation}\label{N31}
T_a =\Op (\sigma)
\quad\text{ with }
\widehat{\sigma}(\theta,\eta) = \chi (\theta,\eta) \widehat{a}(\theta,\eta)
\end{equation}
where $\widehat{a}(\theta,\eta)=\int e^{-ix\cdot\theta}a(x,\eta)\, dx$ is the Fourier transform of $a$ with respect to the first variable 
and where $\chi$ is the cut-off function already introduced to define paraproducts (see~\S\ref{S:3.3}).

Then one may use symbolic calculus for paradifferential operators to factor out the second order elliptic operator 
as the product of two operators of order one. Specifically, we want to get the existence of two symbols 
$\slam,\Slam$ such that 
\begin{equation*}
( \partial_z - T_ \slam ) (\partial_z - T_\Slam)u  \sim 0.
\end{equation*}
To do so, we seek $\slam$ and $\Slam$ under the form  
\begin{equation*}
\slam=\sum\slam_{1-j}, \qquad
\Slam=\sum\Slam_{1-j}
\end{equation*}
(where the indices refer to the order of the symbols) 
such that
\begin{equation*}
\begin{aligned}
&\sum \frac{1}{i^\alpha\alpha!}\ \partial_{\xi}^\alpha a_{m} \, \partial_x^\alpha A_{\ell}
=-\frac{\la \xi\ra^2}{1+|\nabla\eta|^2},\\
&\sum \slam_m+\Slam_m = \frac{1}{1+|\nabla\eta|^2}
\left(  2i \nabla\eta\cdot\xi  + \Delta \eta\right).
\end{aligned}
\end{equation*}
This system can be solved by induction (see~\cite{AM} for the details). 

Then we introduce $ w \defn (\partial_z - T_\Slam)u$ solution to
\begin{equation*}
\partial_z w  - T_{\slam} w \sim 0.
\end{equation*}
Since this is a parabolic evolution equation in $z$ (seen with initial data in $z<0$), we 
conclude that
$$
(\partial_z u  -    T_\Slam u)\arrowvert_{z=0}  =w(0)\sim 0.
$$
This allows to express $\partial_{z}u$ on the boundary $\{z=0\}$ 
in terms of $T_Au$, and hence in terms of tangential derivatives, which is the key ingredient to prove Theorem~\ref{T13}. 

\section{Applications}\label{S:4.1}

Numerous applications of Theorem~\ref{T13} have 
been explored in the analysis of the Cauchy problem for water-waves. 
However, due to the vast scope of this topic, which was extensively covered in lectures by 
Daniel Tataru and Sijue Wu at the Abel Symposium, we refrain from an 
in-depth discussion here. For comprehensive references on this subject, 
we direct the reader to their texts and Chapter 1 in \cite{Berkeley}.

Meanwhile, we highlight only articles which are directly related to the preceding 
paradifferential analysis of the Dirichlet-to-Neumann operator. A series of recent papers delve into the study of the 
Cauchy problem with rough initial data, starting 
with \cite{ABZ1,ABZ3} 
and extending through \cite{ABZ-memoir,dPN1,dPN2,dePoyferre-ARMA,Ai-2019,Ai-2020,shao2023cauchy}. 
Recently, this paradifferential approach experiences 
significant enhancement through the utilization of balanced 
energy estimates introduced by Ai, Ifrim, and Tataru (\cite{ai2019dimensional,ai2020dimensional}).

The paradifferential approach 
has also found application in the examination 
of normal form 
transformations, following the seminal work of Shatah~\cite{Shatah-1985-normal} and Wu~\cite{MR2507638} 
(see~\cite{AlDe,AlDe-Sob,Delort-ICM,MR3784694,Wang-2020,ai2019dimensional,ai2020dimensional,MR3962880,MR3962880,
MR4588319,MR4658635,MR4447560,deng2022wave}). 
Additionally, longstanding questions posed by 
small divisor problems in the theory of standing water waves (see \cite{PlTo,IPT,Icras}) can be analyzed through this perspective (see \cite{AM,AB,BM-2020}). Furthermore, applications extend to the proof of dispersive properties of water-waves (\cite{MR2763354,ABZ1,Ai-2019,Ai-2020}), encompassing results related to control theory~\cite{ABHK,Zhu1}.

In this section, we present two illustrative examples: one concerning the proof of Sobolev energy estimates for the Hele-Shaw equation, and the second regarding the conjugacy of the capillary water-wave equations to a constant coefficient equation.

\subsection{The Cauchy problem for the Hele-Shaw equation}

\begin{theorem}[from \cite{Cheng-Belinchon-Shkoller-AdvMath,Matioc-APDE-2019,AMS,Nguyen-Pausader}]\label{T16}
The Cauchy problem for the Hele-Shaw equation $\partial_t \eta+G(\eta)\eta=0$ is well-posed (localy in time) on the sub-critical Sobolev spaces $H^s(\xT^d)$ for $s>1+d/2$. 
\end{theorem}

Let us explain the strategy of the proof. 
The first step is to quasilinearize the Hele-Shaw equation.  To do so, one uses Lannes' shape derivative foruma (see Proposition~\ref{P:shape}) to prove the following

\begin{lemma}
Let $\phi=\phi(t,x,y)$ denote the harmonic extension of $\eta$ in $\{y<\eta(t,x)\}$ and set
$$
a=1-(\partial_y \phi) \arrowvert_{y=\eta}\quad ,\quad V=-(\nabla_x \phi) \arrowvert_{y=\eta}.
$$
Then
\begin{align}
&\partial_t V+V\cdot\nabla V+aG(\eta)V+\frac{\gamma}{a}V=0\quad\text{where}\label{N21}\\ 
&\gamma=\frac{1}{1+|\nabla \eta|^2}\Big(G(\eta)(a^2+V^2)-2aG(\eta)a-2V\cdot G(\eta)V\Big).\notag
\end{align}
\end{lemma}

Then, one proves that the previous equation for $V$ 
is parabolic. This amounts to prove that the coefficient $a$ is positive.

\begin{lemma}
{$a>0$}.
\end{lemma}
This is the analogue of a famous result by Wu~\cite{WuInvent} for the water wave equations. The proof is readily obtained by noticing that 
the function $y-\phi$ is harmonic and vanishes on the boundary. Therefore, the Hopf-Zaremba principle 
gives $\partial_n (y-\phi)>0$, equivalent to $a>0$.

To analyze \eqref{N21}, we also have to compare the two first order operators: $G(\eta)$ and $A=V\cdot\nabla$. The key point is that, even if $A$ is of order $1$, its real part is of lower order. Indeed
$$
{(A+A^*)f=-(\cnx V)f}  \quad \text{is of order $0$},
$$
where the fact that it is of order $0$ means for instance that it is bounded from $L^2$ to $L^2$. However, this holds only if $V$ is $W^{1,\infty}_x$, which corresponds to the case $s>d/2+2$. To handle the general case $s>d/2+1$, we observe that, for such $s$, $A+A^*$ is not of order $0$ in general, but it is of order strictly smaller than $1$. Namely, it is of order $1-\eps$ provided that 
$V$ is $L^\infty_t(C^{0,\eps})$. This is enough to see $A$ as a subprincipal term: indeed, for a parabolic evolution equation, by using classical interpolation arguments, one can 
see any operator of order strictly smaller than the order of the dissipative term as subprincipal.

The last step of the proof is to paralinearize $G(\eta)$. Here, we cannot rely on Theorem~\ref{T13} since we merely assume that $s>d/2+1$. Instead, we need 
a microlocal  description 
of the Dirichlet-to-Neumann operator 
in the whole range of  $C^s$ domains, $s>1$. The latter analysis was done in \cite{ABZ3}, where 
the following result is proved.

\begin{lemma}
Let~$d\ge 1$ and~$s>1+\frac{d}{2}$. For any $\sigma$ and $\eps$ such that
$$
\frac{1}{2}\leq \sigma \leq s-\mez,\quad 0<\eps\leq \mez, \qquad \eps< s-1-\frac{d}{2},
$$
the remainder $R(\eta)f\defn G(\eta)f-T_\lambda f$ satisfies
\begin{equation*}
\lA R(\eta)f\rA_{H^{\sigma-1+\eps}}\le C\bigl(\| \eta \|_{H^{s}}\bigr)
\lA f\rA_{H^{\sigma}}.
\end{equation*}
Recall that $\lambda$ is defined by \eqref{N30} and $T_\lambda$ by~\eqref{N31}. 
In particular, if $d=1$ then $T_\lambda=\la D_x\ra$ (up to a smoothing operator).
\end{lemma}

Then by combining the previous arguments, one can give a fairly direct proof of 
Sobolev energy estimates (using classical interpolation arguments).

One advantage of this approach is that it gives a 
blow-up criterion (indeed, such blow-up criteria are systematically discussed by Taylor in his books for various nonlinear equations (\cite{MR1121019,MR2744149}) who deduces them from applications of the paradifferential calculus). 
\begin{lemma}\label{P:Cauchy}
Let $\alpha\in (0,1)$. Consider a regular solution $h$ 
as given by Theorem~\ref{T16}, 
and define by $T^*$ its lifespan. Then the following alternative holds:
either $T^*=+\infty$ or
\begin{equation}\label{n110}
\limsup_{t\to T^*}\lA \eta(t)\rA_{C^{1,\alpha}(\xR^d)}=+\infty.
\end{equation}
\end{lemma}

Furthermore, by using completely different techniques (namely using the Boundary Harnack Inequality), in \cite{A-Koch-2023} it is proved that one can control some H\"older norm of the slope in terms of its $L^\infty$-norm and the distance from the initial surface. Since the latter are easily bounded (by maximum principles and comparison arguments), we conclude that the Cauchy problem is well-posed globally in time.

\begin{theorem}[from \cite{A-Koch-2023}]
Consider an integer 
$d\ge 1$ and a real number $s>d/2+1$. 
For any initial data $\eta_0$ in $H^s(\xR^d)$, the Cauchy problem
\begin{equation*}
\partial_{t}\eta+G(\eta)\eta=0,
\quad \eta\arrowvert_{t=0}=\eta_0,
\end{equation*}
has a unique global solution 
\begin{equation*}
\eta\in C^0([0,+\infty);H^s(\xR^d))\cap C^\infty((0,+\infty)\times \xR^d).
\end{equation*}
\end{theorem}
\subsection{Reduction to constant coefficients}

Oversimplifying, one can think of the water-wave equations with surface tension 
as
$$
{Pu=\frac{\partial u} {\partial t } +V(u)\partial_x u
+ i\D^\tq\big(c(u)\D^\tq u\big) = 0} 
$$
where $x\in\xT$ and 
$$
V(u)=\RE \left(\langle D_x\rangle^{-N}u\right)
$$
with $N$ as large as we want (for smooth enough initial data). 
This is a toy model that mimics many features of the 
water-wave equations once written in terms of $\eta$ and the 
good unknown $\omega$ (see~\cite{ABZ1}). 
In this section, the goal is to explain the role of two other transformations (introduced in \cite{AB} and further used in \cite{ABHK}) which can be used 
to further simplify the 
equations.
 
Firstly, one may use a change of variables to flatten the coefficients. Namely, 
we introduce in \cite{AB} a change of variables preserving the $L^2$-norm in $x$, 
of the form
\begin{equation}\label{N20}
h(t,x)\mapsto (1 + \partial_x \kappa(t,x))^{\frac12} \, h(t,x + \kappa(t,x))
\end{equation}
which allows to replace $P$ by 
$$
{Q = \partial_t + W \partial_x + i \D^\tdm + R},\quad\text{$R$ is of order zero}
$$
where one can further assume that $\int_{\xT} W(t,x)\, dx=0$. 
Notice that this is nontrivial since the equation is nonlocal and since it implicitly includes a 
{cancellation} of the term of order $1/2$ (this is the effect of considering 
a change of variables that preserves the $L^2$-norm, namely the 
prefactor $(1 + \partial_x \kappa(t,x))^{\frac12}$ in \eqref{N20}). 

To study ${\partial_t + W\partial_x + i \D^\tdm+R'}$, we then 
seek an operator $A$ such that 
$$
\big[ A,i\D^\tdm\big]+W\partial_x A \quad\text{is a zero order operator}.
$$ 
In \cite{AB}, we find an operator of the form
$$
A=\Op\big(q(t,x,\xi)e^{i\beta(t,x)|\xi|^\mez}\big)
$$
with
$$
\beta=\beta_0(t)+\frac{2}{3}\partial_x^{-1}W.
$$
Then
$$
{\big(\partial_t +W\partial_x+i\D^\tdm\big) A  
=A   \big(   \partial_t +i\D^\tdm+R''\big)}
$$
with $R''$ of order $0$.

It is worth noting that $A$ belongs to $\Op S^0_{\rho,\rho}$ with $\rho=1/2$, known as the exotic class (refer to Coifman and Meyer~\cite{CoMeAs}). 
The fact that $A$ belongs to this symbol class without symbolic calculus reflects the fact that 
the problem is quasi-linear (\cite{Ayman2}). In contrast, for the Benjamin-Ono equation, a similar conjugation is observed with $A\in \Op S^0_{1,0}$.

\section{The Muskat problem}

The Muskat equation serves as a two-fluid model in porous media, analogous to the Hele-Shaw equation. In this context, we examine a time-dependent free surface $\Sigma(t)$ that delineates two fluid domains, denoted as $\Omega_1(t)$ and $\Omega_2(t)$. For simplicity, we restrict our consideration to two-dimensional fluids, so that
\begin{align*}
\Omega_1(t)&=\left\{ (x,y)\in \xR\times \xR\,;\, y>f(t,x)\right\},\\
\Omega_2(t)&=\left\{ (x,y)\in \xR\times \xR\,;\, y<f(t,x)\right\},\\
\Sigma(t)&=\partial \Omega_1(t)=\partial \Omega_2(t)=\{y=f(t,x)\}.
\end{align*}
Let us introduce the density $\rho_i$, the velocity $v_i$, and the pressure $P_i$ in the domain $\Omega_i$ for $i=1,2$. It is assumed that the velocities $v_1$ and $v_2$ adhere to Darcy's law. The equations governing the motion are then given by:
\begin{alignat*}{3}
v_i&=-\nabla (P_i+\rho_i g y) \qquad&&\text{in }&&\Omega_i,\\
\cn v_i&=0 && \text{in }&&\Omega_i,\\
P_1&=P_2 &&\text{on }&&\Sigma,\\
v_1\cdot n &=v_2\cdot n  &&\text{on }&&\Sigma,\\
\partial_t f&=\sqrt{1+(\partial_x f)^2}\, v_2 \cdot n && &&
\end{alignat*}
where $g$ is the gravity and $n$ is the outward unit normal to $\Omega_2$ on $\Sigma$,
$$
n=\frac{1}{\sqrt{1+(\partial_x f)^2}} \begin{pmatrix} -\partial_x f \\ 1\end{pmatrix}.
$$

\subsection{The C\'ordoba-Gancedo formulation}

It has been long recognized that the Muskat problem can be effectively reduced to a parabolic evolution equation for the unknown function $f$ (see~\cite{CaOrSi-SIAM90,EsSi-ADE97,PrSi-book,SCH2004}). The formulation~\eqref{Muskat} used in the subsequent discussion stems from the work of C\'ordoba and Gancedo~\cite{CG-CMP}. In their study of this problem employing contour integrals, they achieved a beautiful reformulation of the Muskat equation in terms of finite differences. 

To be more specific, they derived the following concise formulation of the Muskat equation:
\begin{align}\label{Muskat}
\partial_tf=\frac{\rho}{2\pi}\int_\xR\frac{\partial_x\Delta_\alpha f}{1+\left(\Delta_\alpha f\right)^2}\dalpha,
\end{align}
where $\rho=\rho_2-\rho_1$ is the difference of the densities 
of the two fluids and $\Delta_\alpha f$ is the slope
\begin{align}\label{eq2.2}
\Delta_\alpha f(t,x)=\frac{f(t,x)-f(t,x-\alpha)}{\alpha}\cdot
\end{align}

We assume $\rho_2 > \rho_1$ (which means that the heavier fluid is below the lighter one), allowing us to set $\rho = 2$ without loss of generality. The well-defined nature of the right-hand side is not immediately apparent, and we 
will address this question in the subsequent part of this section.

With the current formulation, it is readily verified that the problem 
is invariant by the following scaling :
\[
f(t,x) \mapsto \frac{1}{\lambda}f\left(\lambda t, \lambda x\right).
\]
Consequently, the two natural critical spaces are the following homogeneous spaces: $\dot{H}^{\frac{3}{2}}(\mathbb{R})$ and $\dot{W}^{1,\infty}(\mathbb{R})$. 

Utilizing this formulation, Cameron~\cite{Cameron} established the existence of a modulus of continuity for the derivative, resulting in a global existence theorem with the sole requirement that the product of the maximal and minimal slopes remains bounded by 1. More recently, Abedin and Schwab also derived a modulus of continuity in \cite{Abedin-Schwab-2020} through Krylov-Safonov estimates. Additionally, C\'ordoba and Lazar, employing a novel formulation of the Muskat equation involving oscillatory integrals \cite{Cordoba-Lazar-H3/2}, demonstrated the global well-posedness of the Muskat equation in time. This holds under conditions that the initial data is suitably smooth and that the $\dot H^{3/2}(\mathbb{R})$-norm is sufficiently small (extended to the 3D case in \cite{Gancedo-Lazar-H2}). We will explain in this section a sharp result, proved with Quoc-Hung Nguyen (\cite{AN3,MR4387237}) which established that the Cauchy problem is well-posed on the endpoint Sobolev space $H^{3/2}(\mathbb{R})$, which is optimal concerning the equation's scaling. Indeed, blow-up results for certain sufficiently large data were demonstrated by Castro, C\'ordoba, Fefferman, Gancedo, and L\'opez-Fern\'andez \cite{CCFG-ARMA-2013,CCFG-ARMA-2016,CCFGLF-Annals-2012}--a contrast to the findings presented in Theorem~\ref{T21} for the Hele-Shaw equation, which is the one-phase version of the Muskat problem.

\subsection{The nonlinearity}

Besides its esthetic aspect, 
the formulation~\eqref{Muskat} allows 
to apply tools at interface of nonlinear partial differential equations and harmonic analysis. 
For instance, one may think of the methods centering around the study of the Hilbert transform~$\mathcal{H}$ 
and Riesz potentials, 
or Besov and Triebel-Lizorkin spaces. 
Recall that the Hilbert transform $\mathcal{H}$ 
is defined by
$$
\widehat{\mathcal{H} u}(\xi)=-i\frac{\xi}{\la \xi\ra} \hat{u}(\xi).
$$
It follows that the fractional Laplacian $\D=(-\partial_{xx})^{\mez}$ satisfies 
$$
\la D_x\ra=\partial_x\mathcal{H}.
$$ 
Alternatively, $\mathcal{H}$ 
can be defined by a singular integral:
\begin{equation}\label{Hilbert:SI}
\mathcal{H}f(x)=\frac{1}{\pi}\mathrm{pv}\int_\xR \frac{f(y)}{x-y}\dy,
\end{equation}
that is
\begin{align*}
\mathcal{H}f(x)
&=\frac{1}{\pi}\mathrm{pv}\int_\xR \frac{f(y)}{x-y}\dy
=\frac{1}{\pi}\lim_{\eps\to 0} \int_{\eps<\la y\ra <\frac1\eps}\frac{f(x-y)}{y}\dy\\
&=-\frac{1}{\pi}
\pv\int_\xR \frac{f(x)-f(x-\alpha)}{\alpha}\dalpha
=-\frac{1}{\pi}
\pv\int_\xR\Delta_\alpha f\dalpha.
\end{align*}
It follows that the fractional Laplacian $\D=\partial_x \mathcal{H}$ satisfies 
$$
\D f =-\frac{1}{\pi}\pv \int_\xR \partial_x\Delta_\alpha f\dalpha.
$$
Consequently,
by writing 
$$
\frac{\partial_x \Delta_\a  f }{1+\left(\Delta_\alpha f\right)^2}=\partial_x\Delta_\a  f-(\partial_x\Delta_\a  f)\frac{(\Delta_\a f)^2 }{1+\left(\Delta_\alpha f\right)^2},
$$
we see that the Muskat equation can be written under the form
$$
\partial_tf+\D f=\mathcal{T}(f)f\quad \text{where}\quad \mathcal{T}(f)f=-\frac{1}{\pi}\, \int_\xR(\partial_x\Delta_\alpha f)\frac{(\Delta_\a f)^2 }{1+\left(\Delta_\alpha f\right)^2}\dalpha.
$$
The next proposition states 
that the map $f\mapsto \mathcal{T}(f)f$ is locally Lipschitz from 
$H^{\tdm}(\xR)$ to $L^2(\xR)$.

\begin{proposition}\label{P:continuity}
Consider the operator
$$
\mathcal{T}(f)g=-\frac{1}{\pi}\int_\xR \Delta_\alpha g_x \ \frac{(\Delta_\alpha f)^2}{1+(\Delta_\alpha f)^2}\dalpha,
$$
where $g_x\defn \partial_x g$.
\begin{enumerate}
\item\label{Prop:low1} For all $f$ in $H^{1}(\xR)$ and all $g$ in $H^{\tdm}(\xR)$, 
the function
$$
\alpha\mapsto \Delta_\alpha g_x \ \frac{(\Delta_\alpha f)^2}{1+(\Delta_\alpha f)^2}
$$
belongs to $L^1_\alpha(\xR;L^2_x(\xR))$. Consequently, $\mathcal{T}(f)g$ belongs to $L^2(\xR)$. 
Moreover, there is a constant $C$ such that
\begin{equation}\label{nTL2}
\lA \mathcal{T}(f)g\rA_{L^2}\le C \lA f\rA_{\dot{H}^1}\lA g\rA_{\dot{H}^\tdm}.
\end{equation}

\item \label{Prop:low3} There exists a constant $C>0$ such that, 
for all functions $f_1,f_2$ in $H^{1}(\xR)$ and for all 
$g$ in $H^\tdm(\xR)$,
$$
\lA (\mathcal{T}(f_1)-\mathcal{T}(f_2))g\rA_{L^2}
\le C \lA f_1-f_2\rA_{\dot{H}^{1}}\lA g\rA_{\dot{H}^{\tdm}}.
$$

\item \label{Prop:low2} The map $f\mapsto \mathcal{T}(f)f$ is locally Lipschitz from 
$H^{\tdm}(\xR)$ to $L^2(\xR)$.
\end{enumerate}
\end{proposition}

The previous proposition is 
an example of properties that are very simple to prove using 
the definition of Sobolev spaces in terms of finite differences. In particular, it relies on the fact that, 
for $s\in (0,1)$, 
the homogeneous Sobolev norm $\lA\cdot\rA_{\dot{H}^s}$ is equivalent to 
the Gagliardo semi-norm
\begin{equation}\label{Gagliardo}
\lA u\rA_{\dot H^{s}}^2\defn \frac{1}{4\pi c(s)}
\underset{\xR\times\xR}\iint \frac{\la u(x)-u(y)\ra^2}{\la x-y\ra^{2s}}\frac{\dx\dy}{\la x-y\ra}
\quad\text{with}\quad 
c(s)=\int_\xR \frac{1-\cos(h)}{\la h\ra^{1+2s}}\dh.
\end{equation}

Using the previous point of view, one can give an elementary proof of 
a paralinearization formula for the nonlinearity $\mathcal{T}(f)f$. 

The goal here is 
decompose the nonlinearity 
 into several pieces having different roles.  Since it is a problem in fluid dynamics, one expects to extract from the nonlinearity at least two terms: 
 a convective term of the form $V\partial_x f$ and 
 an elliptic component of the form $\gamma \D f$, that is an equation of the form
$$
{\partial_t f+V\partial_x f+\gamma \D f=R(f)}
$$
for some coefficients $V$ and $\gamma$ depending on $\partial_x f$.

To reach this goal, a standard strategy is to use a paradifferential analysis. 
For the Muskat equation, 
 this idea was implemented independently in~\cite{A-Lazar} 
 and \cite{Nguyen-Pausader}. These two approaches are very different: the approach in \cite{Nguyen-Pausader} applies for many physical equations, while 
 the approach in \cite{A-Lazar} is adapted to study critical problem, as we discuss below.

To do this, following {Shnirelman} (\cite{Shnirelman}),  
we consider a simpler version of paraproducts.

\begin{proposition}[\cite{A-Lazar}]
\label{P:A-Lazar}
Let $0<\eps<1/2$. Given a bounded function $a=a(x)$, denote by 
$\tilde{T}_a\colon {H}^{1+\eps}(\xR)\rightarrow {H}^{1+\eps}(\xR)$ the paraproduct-type operator 
defined by
$$
{\tilde{T}_ag=\Lambda^{-(1+\eps)}(a\Lambda^{1+\eps}g)} \quad, \quad \Lambda=(I-\partial_{xx})^{1/2}.
$$
Then
$$
{\mathcal{T}(f)f=\tilde{T}_{\frac{f_x^2}{1+f_x^2}}\la D_x\ra f+\tilde{T}_{V(f)}\partial_x f+
R_\eps(f)}
$$
where $R_\eps$ is of order $<1$:
$$
\lA R_\eps(f)\rA_{H^{1+\eps}}\le \mathcal{F}\big(\lA f\rA_{{H}^{\tdm+\eps}}\big)
\lA f\rA_{{H}^{\frac{3}{2}+\eps}}
\lA f \rA_{H^{2+{\frac{\eps}{2}}}}.
$$
\end{proposition}

To prove this proposition, 
the approach in~\cite{A-Lazar} 
exploits the formulation~\eqref{Muskat} of the Muskat equation to 
 give such a paradifferential decomposition in a direct manner. 
The desired decomposition of the nonlinearity is achieved by 
splitting the coefficient
into its odd and even components: 
\begin{align*}
\mathcal{O}\left(\alpha,\cdot\right)
&= \frac{1}{2}\bigg(\frac{1}{1+\left(\Delta_\alpha f\right)^2} -\frac{1}{1+\left(\Delta_{-\alpha} f\right)^2}\bigg),\\
\mathcal{E}\left(\alpha,\cdot\right) &= \frac{1}{2}\bigg(\frac{1}{1+\left(\Delta_\alpha f\right)^2} +
\frac{1}{1+\left(\Delta_{-\alpha} f\right)^2}\bigg).
\end{align*}

Since $\Delta_\alpha f(x)\to f_x(x)$ when $\alpha \to 0$, 
we further decompose $\mathcal{E}\left(\alpha,\cdot\right)$ 
as follows
$$
{\mathcal{E}\left(\alpha,\cdot\right) =\frac{1}{1+(\partial_xf)^2}+\left(\mathcal{E}\left(\alpha,\cdot\right) -\frac{1}{1+(\partial_xf)^2}\right)}.
$$
By considering that ${\displaystyle {\D f= -\frac{1}{\pi}\int_\xR\partial_x\Delta_\alpha f\dalpha}}$ 
we obtain the following decomposition of the nonlinearity:
\begin{equation}\label{newMuskat}
{\partial_t f+\frac{1}{1+(\partial_xf)^2}\D f+V(f)\partial_x f+R=0}
\end{equation}
with
\begin{align*}
V=-\frac{1}{\pi}\int_\xR\frac{\mathcal{O}\left(\alpha,\cdot\right)}{\alpha}\dalpha,\qquad 
R=-\frac{1}{\pi}\int_\xR \frac{f_x(\cdot-\alpha)}{\alpha} \left(\frac{1}{1+(\Delta_\alpha f)^2}-\frac{1}{1+f_x^2}\right)\dalpha.
\end{align*}

\subsection{Estimate of the critical norm}

Proposition~\ref{P:A-Lazar} enables the establishment of well-posedness for the Cauchy problem on sub-critical Sobolev spaces. However, to address the Cauchy problem on critical spaces, additional ideas are requisite.

The first estimation of the critical Sobolev $H^{3/2}$-norm was obtained in the work of C{\'o}rdoba and Lazar~(\cite{Cordoba-Lazar-H3/2}), who applied methods from harmonic analysis to establish the well-posedness of the Cauchy problem for certain initial data, where the slope can be arbitrarily large. Specifically, they proved the following result.

\begin{theorem}[C\'ordoba and Lazar] The Cauchy problem for the Muskat equation is well-posed globally in time for all 
initial data $f_0\in H^{\frac52}(\xR)$ such that 
$$
(1+\lA \partial_x f_0\rA_{L^\infty}^4)\lA f_0\rA_{\dot{H}^{\tdm}}\ll 1.
$$
\end{theorem}
Their main technical result is a key {\em a priori} estimate 
in the critical Sobolev space, which reads:
$$
{\fract \lA f\rA_{\dot H^{\frac{3}{2}}}^2
+ 
{\int_\xR\frac{( \partial_{xx}f)^2}{1+(\partial_x f)^2}\dx}}
\lesssim  \left(\lA f\rA_{\dot H^{\frac{3}{2}}}
+\lA f\rA_{\dot H^{\frac{3}{2}}}^2\right)
\lA f\rA_{\dot{H}^2}^2.
$$
While the statement of this estimate is straightforward, its proof is very intricate due to the non-Banach algebra structure of $H^\frac{1}{2}(\mathbb{R})$. C{\'o}rdoba and Lazar addressed this challenge in their proof, which relies on a clever reformulation of the equation in terms of oscillatory integrals and the systematic application of Besov spaces techniques. Further insights into this approach can be found in the works of Gancedo-Lazar~\cite{Gancedo-Lazar-H2} and Granero-Bellinch\'on and Scrobogna~\cite{GBS-2020}.

The question then arises of how to derive a global existence result from their inequality. The underlying idea is quite straightforward: the objective is to absorb the right-hand side with the left-hand side. This necessitates controlling the denominator $1+(\partial_xf)^2$ from below. To achieve this, an integration by parts argument is employed, tracing back to the work of C\'ordoba and Gancedo~\cite{CG-CMP2} 
(see also Cameron~\cite{Cameron} 
and 
Gancedo-Lazar~\cite{Gancedo-Lazar-H2}, among others). 
This yields
\[
\fract \lVert \partial_x f \rVert_{L^\infty} \lesssim \lVert f \rVert_{H^2}^2.
\]
By combining these two inequalities, it follows that there exists $c_0>0$ small enough, so that
\[
\fract \lVert f \rVert_{\dot H^{\frac{3}{2}}}^2 + c_0\int_{\mathbb{R}}  (\partial_{xx}f)^2 \,dx \le 0,
\]
thereby estimating the $L^\infty_t(H^{3/2}_x)$-norm.

We proceed to discuss two improvements: global existence for small data, without the need for initial data to be smooth, and secondly, local well-posedness for large data in the critical space. These improvements are based on two distinct techniques.

\subsection{A null-type property}

In \cite{AN3,MR4387237}, the first 
idea is to prove a (weak) null-type property. 
By this we mean the proof of an inequality similar to C{\'o}rdoba-Lazar 
but with the 
key difference that the right-hand side involves the 
weak coercive quantity given by the left-hand side.

\begin{theorem}[\cite{AN3}] 
The Cauchy problem is well-posed, globally in time, for all initial data {$f_0\in H^{\frac32}(\xR)$} such that ${\lA f_0\rA_{\dot{H}^{\tdm}}\ll 1.}$
\end{theorem}

The proof relies on a 
null-type property, enabling to prove that
$$
\fract \lA f\rA_{\dot H^{\frac{3}{2}}}^2
+ {
{{\int_\xR \frac{( \partial_{xx}f)^2}{1+(\partial_x f)^2} \dx}}}
\lesssim
\left(1+\Vert f\Vert_{\dot H^{\frac{3}{2}}}^{7}\right)
\Vert f\Vert_{\dot H^{\frac{3}{2}}}
{{\int_\xR \frac{( \partial_{xx}f)^2}{1+(\partial_x f)^2} \dx}.}
$$
This estimate implies at once a global well-posedness result for small data. 
Indeed, this part is perturbative in character, and one can absorb the right-hand side by the left-hand side, 
provided that $\Vert f\Vert_{\dot H^{\frac{3}{2}}}$ is small enough: 
if $\Vert f\Vert_{\dot H^{\frac{3}{2}}}$ is small enough then
$$
\fract \lA f\rA_{\dot H^{\frac{3}{2}}}^2
+c_0\int_\xR \frac{( \partial_{xx}f)^2}{1+(\partial_x f)^2} \dx\le 0.
$$

\subsection{Paralinearization and enhanced regularity}

The main result in \cite{AN3} is

\begin{theorem}[\cite{AN3}]
The Cauchy problem is locally well-posed for any initial data in the critical space $H^{\frac32}(\xR)$.
\end{theorem}

To prove this result, we do not estimate the $H^\tdm(\xR)$ norm, but a quantity which depends on the initial data. 
The idea is to exploit an enhanced decay of the Fourier transform. Specifically, 
we use the fact that, 
for all $f_0\in H^\tdm(\xR)$ there is a weight $\kappa_0$ such that
\begin{align*}
&{\lA f_0\rA_{\mathcal{H}^{\tdm,\kappa_0}}^2\defn\int_{\xR} (1+\la\xi\ra)^3 \kappa_0(|\xi|)^2 \big\vert\hat{f_0}(\xi)\big\vert^2\dxi <+\infty}
\quad\text{and}\\
&
{\lim_{|\xi|\to+\infty}\kappa_0(|\xi|)=+\infty}.
\end{align*}
We then estimate the $L^\infty_t(\mathcal{H}^{\tdm,\kappa_0})$-norm of $f$ using the previous decomposition~\eqref{newMuskat} of the nonlinearity, by 
\begin{itemize}
\item commuting $\D \kappa_0(D_x)$ with the equation \eqref{newMuskat}
\item taking the $L^2_x$ scalar product with 
$\D^{2}\kappa_0(D_x)f$ and integrating in time.
\end{itemize}
The proof is very technical and we refer to the introduction of the 
original article \cite{AN3} for more explanations.

\end{document}